\documentclass{article}  
\usepackage[preprint, nonatbib]{neurips_2022}

\usepackage[utf8]{inputenc} 
\usepackage[T1]{fontenc}    
\usepackage[]{hyperref}       
\usepackage{url}            
\usepackage{booktabs}       
\usepackage{amsfonts}       
\usepackage{nicefrac}       
\usepackage{microtype}      
\usepackage{xcolor}         
\usepackage{amsmath, amssymb, graphicx, url, mathtools, empheq}
\usepackage[toc,page,header]{appendix}
\usepackage{makecell}
\usepackage[misc]{ifsym}
\usepackage[backend=bibtex,maxbibnames=3,url=false,doi=false,eprint=false,isbn=false]{biblatex}
\AtEveryBibitem{\clearfield{pages}} 

\usepackage{setspace}
\usepackage[ruled]{algorithm2e}
\usepackage{amssymb}

\usepackage{amsthm}
\usepackage{amsmath}
\usepackage{amsfonts}
\usepackage{bbm}
\usepackage{xspace}
\usepackage{tikz}
\usepackage{graphbox} 
\usepackage{bbold} 
\usepackage[normalem]{ulem}
\usepackage{wrapfig}
\usepackage[export]{adjustbox} 
\usepackage{float}
\usepackage{hyperref}
\usepackage{xcolor}
\usetikzlibrary{decorations.pathreplacing}
\usepackage{cleveref}

\crefname{equation}{equation}{}
\Crefname{equation}{Equation}{}
\crefrangelabelformat{equation}{(#3#1#4--#5#2#6)}

\crefmultiformat{equation}{(#2#1#3}{, #2#1#3)}{#2#1#3}{#2#1#3}
\Crefmultiformat{equation}{(#2#1#3}{, #2#1#3)}{#2#1#3}{#2#1#3}

\newif\ifcomment
\commenttrue

\newcommand{\pemp}{\ensuremath{\hat{P}_N}\xspace} 

\newcommand{\hnorm}[1]{\|{#1}\|_{\mathcal H}}

\newcommand{\rkhs}{{\mathcal H}\xspace}

\newcommand{\CVaR}{\operatorname{CVaR}}
\newcommand{\MMD}{\operatorname{MMD}}

\newtheorem{proposition}{Proposition}[section]
\newtheorem{theorem}{Theorem}[section]
\newtheorem{lemma}[proposition]{Lemma}

\newtheorem*{remark}{Remark}

\newcommand{\gh}{\hat{g}}
\newcommand{\xh}{\hat{x}}

\newcommand{\E}{\mathcal E_N}

\newcommand{\Etrue}{\mathbb{E}_{P_0}}
\newcommand{\Ehat}{\mathbb{E}_{\hat{P}_N}}

\newcommand{\oneOverRootN}{\mathcal{O}(\frac{1}{\sqrt{N}})\xspace}

\newcommand{\confidTerm}{M\sqrt{\frac{2\log(1/\delta)}{N}}\xspace}

\newtheorem{corollary}{Corollary}[theorem]

\title{Maximum Mean Discrepancy Distributionally Robust\\
Nonlinear Chance-Constrained Optimization with Finite-Sample Guarantee}

\author{%
  \fontsize{8.7}{8.7}\selectfont Yassine Nemmour\textsuperscript{1,*}~~~Heiner Kremer\textsuperscript{1,*}~~~Bernhard Sch\"olkopf\textsuperscript{1}~~~Jia-Jie Zhu\textsuperscript{2}\\
  \fontsize{8.3}{8.3}\selectfont \!\!\!\textsuperscript{1}MPI for Intelligent Systems, Tübingen~~~\textsuperscript{2}Weierstrass Institute for Applied Analysis and Stochastics, Berlin\\
  \fontsize{8.3}{8.3}\selectfont \!\!\!\textsuperscript{*}Equal Contribution~~~\textsuperscript{\Letter}\{\texttt{ynemmour@tuebingen.mpg.de}\}\\
}

\begin{document}

\maketitle

\begin{abstract}
This paper is motivated by addressing open questions in distributionally robust chance-constrained programs (DRCCP) using the popular Wasserstein ambiguity sets.
Specifically, the computational techniques for those programs typically place restrictive assumptions on the constraint functions
and the size of the Wasserstein ambiguity sets is often set using costly cross-validation (CV) procedures or conservative measure concentration bounds.
In contrast, we propose a practical DRCCP algorithm using kernel \emph{maximum mean discrepancy} (MMD) ambiguity sets, which we term MMD-DRCCP, to treat general nonlinear constraints without using ad-hoc reformulation techniques. 
MMD-DRCCP can handle general nonlinear and non-convex constraints with a proven finite-sample constraint satisfaction guarantee
of a dimension-independent $\oneOverRootN$ rate, achievable by a practical algorithm.
We further propose an efficient bootstrap scheme for constructing sharp MMD ambiguity sets in practice without resorting to CV.
Our algorithm is validated numerically on a portfolio optimization problem and a tube-based distributionally robust model predictive control problem with non-convex constraints.
\end{abstract}
\section{Introduction}
Chance-constrained programs (CCP) frame optimization problems under uncertainty using soft probabilistic constraints. Compared to their robust optimization counterparts, CCP solutions generally satisfy the constraints with high probability in practice while maintaining performance.
Computing an exact solution of a CCP requires knowledge of the underlying probability distribution, which is generally unavailable or uncertain.
For this reason, CCP has recently been combined with distributionally robust optimization~\cite{delageDistributionallyRobustOptimization2010}~(DRO) using Wasserstein ambiguity sets \cite{MohajerinEsfahani2018,gaoDistributionallyRobustStochastic2016}, as proposed by works such as \cite{Weijun2018,hotaDataDrivenChanceConstrained2018}.
DRO addresses uncertain optimization problems by finding a solution for the worst-case distribution from a set of distributions, the ambiguity set.
In this paper, we will focus on ambiguity sets constructed using the kernel maximum mean discrepancy.
Enabled by recent advances in applied mathematics, particularly in variational analysis and probability theory, the (Kantorovich-)Wasserstein distances have become a useful class of probability metrics used in operations research, computer graphics, machine learning, and numerous scientific fields.
We refer the interested reader to \cite{santambrogioOptimalTransportApplied2015} for the theory of optimal transport.
Within the Wasserstein ambiguity set DRCCP framework, with the notable exceptions of \cite{hotaDataDrivenChanceConstrained2018,guDistributionallyRobustChanceConstrained2021}, most works limit the class of constraint functions to be affine
in the uncertainty variable \cite{Weijun2018, chenDataDrivenChanceConstrained, ho-nguyenDistributionallyRobustChanceConstrained2020}.
This is a severe limitation for applying their methods in practice.
For convenience, we give an overview of DRCCP works in Table~\ref{tab:drccp_comp} to highlight the gap that this paper fills in terms of handling nonlinear constraints.
In addition, many works rely on cross-validation (CV) procedures to set the size of the ambiguity set.
CV is computationally prohibitive and thus seldomly used in large-scale tasks, such as deep learning and tasks involving complex simulations.

\begin{table}[!t]
    \centering
    \begin{tabular}{c|c}
        Constraint function w.r.t. $\xi$ & Approach 
        \\
        \hline
         Affine & MIP: \cite{Weijun2018, ho-nguyenDistributionallyRobustChanceConstrained2020, chenDataDrivenChanceConstrained, jiDataDrivenDistributionallyRobust}, CVaR: \cite{Weijun2018, hotaDataDrivenChanceConstrained2018, guDistributionallyRobustChanceConstrained2021} 
         \\
         Quadratic & CVaR: \cite{guDistributionallyRobustChanceConstrained2021} 
         \\
         Concave & Cutting-plane algorithm: \cite{hotaDataDrivenChanceConstrained2018} 
         \\
         \makecell{Convex (known Lipschitz \\constant for every $x$)} & Convex inner approximation: \cite{hotaDataDrivenChanceConstrained2018}
         \\
         General nonlinear & \textbf{This work}
    \end{tabular}
    \caption{Comparison of recent DRCCP works using Wasserstein ambiguity sets and our work.
    }
    \label{tab:drccp_comp}
\end{table}

In order to address those open questions, this paper proposes DRCCP with general nonlinear constraints based on the maximum mean discrepancy (MMD) ambiguity sets introduced shortly, which we term (MMD-DRCCP). Different from DRCCP based on Wasserstein distances, our methodology does not rely on ad-hoc reformulation techniques that are highly dependent on the constraint function classes and their specific closed-form support functions, if they exist.
Instead, we use Hilbert spaces generated using expressive kernels as universal function approximators to treat general nonlinear constraints with a single unified reformulation technique.

\noindent
\textbf{Contributions and results overview:}
\begin{enumerate}
    \item We propose the MMD-DRCCP in Section~\ref{sec:main}.
    Following that, we derive (a) an exact reformulation of the MMD-DRCCP employing the recent advances in kernel methods for robust machine learning \cite{zhuKernelDistributionallyRobust2020} and (b) a convexity-preserving conditional Value-at-Risk (CVaR) approximation~\eqref{eq:approx_sem_inf} that can treat general nonlinear constraints.
    \item We show that, in contrast to Wasserstein ambiguity sets, MMD ambiguity sets can be constructed in a simple and efficient way using a bootstrap procedure (Algorithm~\ref{alg:bootstrap}) or using a computable estimation error bound.
    \item 
    To justify our approximation scheme~\eqref{eq:approx_sem_inf},
    we then give a finite-sample constraint satisfaction guarantee for general nonlinear DRCCPs in Section~\ref{sec:guarantee}.
    Informally, it certifies that, despite not using the so-called exact reformulation of DRCCP, our finite-sample solution satisfies the population version of the distributionally robust CVaR constraints approximately with an error of order $\oneOverRootN$ independent of the problem dimensions.
    \item In Section~\ref{sec:num_exp}, we validate our methods using numerical examples in operations research and control problems.
\end{enumerate}

\section{MMD-DRCCP}
\label{sec:main}
In this section, we introduce DRCCPs with MMD ambiguity sets.
We discuss how, in contrast to Wasserstein ambiguity sets, we can construct MMD ambiguity sets in a principled way using a practical bootstrap approach. We derive an exact reformulation of our problem based on the strong duality result of \cite{zhuKernelDistributionallyRobust2020} and provide a tractable CVaR relaxation. 
Before formally introducing the problem, we provide necessary background material on reproducing kernel Hilbert spaces.

\subsection{Reproducing Kernel Hilbert Spaces}
\label{subsec:rkhs}
A kernel is a similarity measure defined by a symmetric function $k: \mathcal X \times \mathcal X \rightarrow \mathbb R$ and said to be positive definite (PD) if $\sum_{i=1}^n \sum_{j=1}^n a_j a_i k(x_i, x_j) \geq 0$ for any $n\in \mathbb N$, $\{x_i\}_{i=1}^n \subset \mathcal X$ and $\{a_i\}_{i=1}^n \subset \mathbb R$. For every PD kernel there exists a feature map $\phi : \mathcal X \rightarrow \mathcal H$ taking values in a reproducing kernel Hilbert space (RKHS) $\mathcal H$ such that $k(x, y) = \langle \phi(x), \phi(y)\rangle_{\mathcal H}$, where $\langle \cdot, \cdot \rangle_\mathcal{H}$ denotes an inner product on $\mathcal H$. 
The inner product induces a norm via $||f||_\mathcal{H} \coloneqq \sqrt{\langle f, f\rangle_\mathcal{H}}$.
For a given distribution $P$ one denotes the kernel mean embedding (KME) as $\mu_P \coloneqq \int k(x, \cdot) dP$. Equipped with these tools, one can introduce a metric between two distributions $P$ and $Q$ as $||\mu_P - \mu_Q||_\mathcal{H}$, known as the maximum mean discrepancy (MMD) \cite{grettonKernelTwoSampleTest2012}. The reproducing property of an RKHS enables one to rewrite this metric as
$\MMD(P, Q) = \mathbb E_{x, x' \sim P} k(x, x') + \mathbb E_{y, y' \sim Q}k(y, y') - 2\mathbb E_{x\sim P, y\sim Q}k(x, y)$.
This closed-form expression is a definitive advantage of MMD since computing Wasserstein distances is intractable in general.
Notably, the MMD and the 1-Wasserstein distance both belong to the integral probability metric family, while an entropic regularized optimal transport metric can also be interpreted as an MMD\cite{feydyInterpolatingOptimalTransport2018}.
We refer to standard texts \cite{berlinetReproducingKernelHilbert2011,scholkopf2002learning,wendlandScatteredDataApproximation2004,smola2007hilbert,steinwartSupportVectorMachines2008} for comprehensive introductions to kernel methods.

\subsection{Chance constraint programs with MMD ambiguity sets}
\label{subsec:prob}
Next, we formally introduce CCPs and DRCCPs.
For simplicity, we restrict our attention to scalar-valued functions in this paper.
Let $f: \mathcal{X} \subset \mathbb{R}^n \times \Xi \subseteq \mathbb R^m \rightarrow \mathbb{R}$ denote a function that defines an uncertain inequality constraint $f(x, \xi) \leq 0$ depending on a random variable $\xi \in \Xi$.
In particular,
we do not exclude the possibility that $f$ may be nonlinear, non-convex, or semi-continuous.
We consider the linear cost $c^Tx$, with $c \in \mathbb R^n$, without loss of generality. A CCP with risk level $\alpha \geq 0$ is then defined as
\begin{equation}
\label{eq:ccp}
\begin{alignedat}{1}
    & \min_{x \in \mathcal{X}} c^T x\\
    & \text{subject to }P_0 [f(x, \xi) \leq 0] \geq 1 - \alpha.
\end{alignedat}
\end{equation}
This has the interpretation that the inequality constraints can be violated with probability at most $\alpha$.
Since generally the underlying data distribution $P_0$ is unknown and one often only has access to a sample from it, we expand this formulation to its distributionally robust counterpart. 
We consider a worst-case distribution within a set of plausible distributions, the so-called ambiguity set, and define the DRCCP as
\begin{equation}
    \label{eq:drcc}
\begin{alignedat}{1}
    & \min_{x \in \mathcal{X}} c^T x\\
    & \text{subject to }    \inf_{P \in \mathcal P} P [f(x, \xi) \leq 0] \geq 1 - \alpha.
\end{alignedat}
\end{equation}
We construct a maximum mean discrepancy (MMD)-based ambiguity set as a ball of radius $\varepsilon$ centered at the empirical distribution $\pemp$. Given samples $\{\xi\}_{i=1}^N$ of the true distribution $P_0$, the empirical distribution is then given by $\pemp = \sum_{i=1}^N \delta_{\xi_i}$.
As the sample size goes to infinity, the empirical distribution $\pemp$ of the sample converges to the true distribution by the weak law of large numbers. Thus, the radius $\varepsilon$ should be chosen in a data-driven way that reflects the confidence according to the sample size $N$.
In the rest of the paper, we denote the MMD ambiguity set by
\begin{equation}
\label{eq:amb_set}
\mathcal P \coloneqq \{P : \MMD(P, \pemp) \leq \varepsilon\}.
\end{equation}
If $\varepsilon$ is chosen large enough such that the true distribution is contained in $\mathcal{P}$, then the solution of the DRCCP \eqref{eq:drcc} will also satisfy the constraint of the original CCP \eqref{eq:ccp}.
Notably, MMD ambiguity sets enable us to set the ambiguity radius a priori in a few simple-yet-principled ways not available for the Wasserstein counterpart.
Using the estimation error bound of MMD estimators \cite{Tolstikhin2017,grettonKernelTwoSampleTest2012}, we have that, with probability $1-\delta$, the population distribution $P_0$ is contained in an MMD ball around the empirical distribution
\begin{equation}
    \label{eq:rate}
    \MMD(P_0, \pemp) \leq \sqrt{\frac{C}{N}} + \sqrt{\frac{2C\log(1/\delta)}{N}},
\end{equation}
where $C$ is a constant such that $\sup_x k(x, x) \leq C \leq \infty$. For the common Gaussian kernel, $C=1$. Note that \eqref{eq:rate} is dimension-free.
With this result, we can simply set the radius of the MMD ambiguity set to the RHS of \eqref{eq:rate}.
In practice, however, concentration bounds such as \eqref{eq:rate} are overly conservative \cite{grettonKernelTwoSampleTest2012}.
We now propose an MMD bootstrap scheme to obtain tighter confidence intervals, which is not available for Wasserstein distances.

\subsection{Bootstrap construction of MMD ambiguity sets}
\begin{figure}[t!]
    \centering
    \includegraphics[scale=0.85]{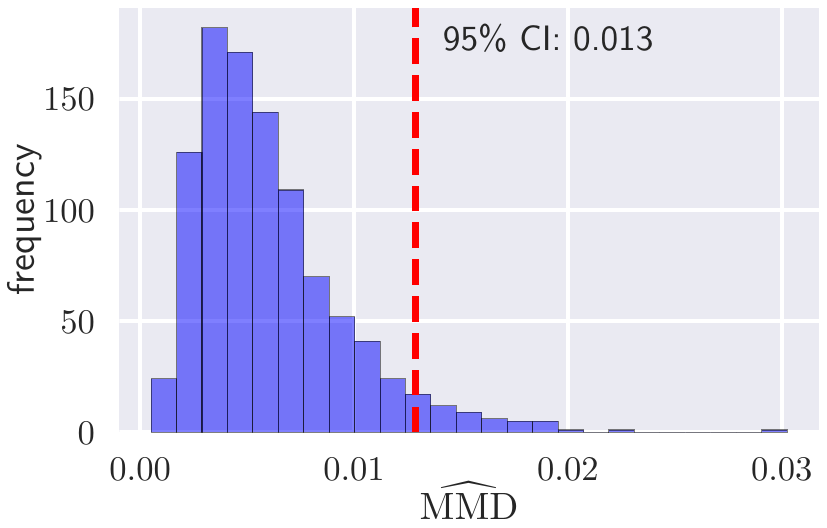}
    \caption{Bootstrap construction of the MMD ambiguity set. We exemplarily sample $N=100$ points from a standard normal distribution and compute bootstrap estimates of $\operatorname{MMD}(P_0,\pemp)$ over $B=1000$ bootstrap samples using Algorithm~\ref{alg:bootstrap}. We set the radius of the ambiguity set to the $\beta = 95\%$ confidence bound $\varepsilon = 0.013$. For comparison, using instead $10000$ additional samples from $P_0$ to estimate $\operatorname{MMD}(P_0,\pemp)$ yields $\varepsilon = 0.010$. We conclude that in this case with high probability the true distribution is contained in our bootstrap MMD ambiguity set.}
    \label{fig:bootstrap}
\end{figure}
We propose to construct the bootstrap MMD ambiguity set in a similar fashion as for the two-sample test in \cite{grettonKernelTwoSampleTest2012}, based on the results of \cite{bootstrap-CI} for general degenerate V-statistics. 
Let $\{ \tilde{\xi}_i \}_{i=1}^N$ denote a bootstrap sample of $\pemp$ with distribution $\tilde{P}$, i.e., drawn with replacement from $\{ \xi_i \}_{i=1}^N$.
Define the (biased) MMD estimator as $\widehat{\operatorname{MMD}}(\tilde{P},\hat{P}_N) = \sum_{i,j=1}^N k(\xi_i,\xi_j) + k(\tilde{\xi}_i, \tilde{\xi}_j) - 2 k(\xi_i, \tilde{\xi}_j)$. Then by the weak law of large numbers and the bootstrap result for V-statistics of \cite{bootstrap-CI} we have that $\widehat{\operatorname{MMD}}(\tilde{P},\pemp) \overset{d}{\rightarrow} \operatorname{MMD}(P_0,\pemp)$ as $N \rightarrow \infty$. 
For a fixed confidence level $\beta$, this lets us determine the radius $\varepsilon$ of the uncertainty set as the $\beta$-quantile of the bootstrap distribution $\widehat{\operatorname{MMD}}(\tilde{P}, \hat{P}_N)$ (see Figure~\ref{fig:bootstrap}). Details on the procedure can be found in Algorithm~\ref{alg:bootstrap}.
We emphasize that bootstrap techniques for MMD have been used in large-scale machine learning tasks for high-dimensional data and are not available for Wasserstein ambiguity sets due to the lack of closed-form estimators.

\begin{algorithm}[h!]
\caption{Bootstrap MMD ambiguity set}
\label{alg:bootstrap}
\KwData{Sample $\{\xi\}_{i=1}^N$, Number of bootstrap samples $B$, Confidence level $\beta$}
\KwResult{Radius of MMD ambiguity set $\varepsilon$}
$K \gets kernel(\xi, \xi)$\; 
\For{$m = 1,\ldots , B$}{
  Draw a set $I$ of $N$ numbers from $\{1,\ldots,N\}$ with replacement\;
  $K_x \gets \sum_{i,j=1}^N K_{ij}$;  \ \
  $K_y \gets \sum_{i,j \in I} K_{ij}$\;
  $K_{xy} \gets \sum_{j \in I} \sum_{i=1}^N K_{ij} $\;
  $\text{MMD}[m] \gets \frac{1}{N^2} (K_x + K_y - 2 K_{xy})$\;
}
$\text{MMD} \gets sort(\text{MMD})$\;
$\varepsilon \gets \text{MMD}[ceil(B \beta)]$\;
\end{algorithm}

\subsection{Exact reformulation}
\label{subsec:exact}
Given the MMD ambiguity set $\mathcal{P}$, we denote the feasible region $Z$ of the DRCCP \eqref{eq:drcc} as
\begin{equation}
    \label{eq:set_z}
    Z \coloneqq \big\{ x \in \mathcal{X}: \inf_{P \in \mathcal P} P \{ f(x, \xi) \leq 0\} \geq 1 - \alpha \big\},
\end{equation}
where $\mathcal P $ is the MMD ambiguity set defined in \eqref{eq:amb_set}.
Since for MMD ambiguity sets, the infimum in \eqref{eq:set_z} is challenging to compute, we first summon the strong duality result proved in \cite{zhuKernelDistributionallyRobust2020} to embed the chance constraint into an RKHS. With the dual form of the constraint, the DRCCP becomes a kernel machine learning problem in finding an RKHS function that majorizes $1(f(x,\xi) \leq 0)$, where $1$ denotes the indicator function. We visualize this idea in Figure~\ref{fig:majorization} and formalize it in the following Proposition.
\begin{proposition}
\label{thm:exact}
We consider the feasible set $Z$ in \eqref{eq:set_z} with the uncertain nonlinear constraint $f(x, \xi)\leq 0$, the risk-level $\alpha$ and the RKHS function $g(\xi) \in \mathcal H$. The decision variable is denoted by $x$ and $\xi$ is the uncertain variable.
Then the feasible set $Z$ can be reformulated as 
\begin{subequations}
    \label{eq:exact_ref}
    \begin{empheq}[left=Z \coloneqq \left\{ x \in \mathcal{X} :, right=\right\} ]{align}
        &g_0 + \frac{1}{N} \sum_{i=0}^N g(\xi_i) + \varepsilon ||g||_{\mathcal H} \leq \alpha \label{eq:cc_risk}\\
        &\mathbb 1(f(x, \xi) > 0) \leq g(\xi) + g_0 \, \ \  \forall \xi \in \Xi \label{eq:characteristic}\\
        &g \in \mathcal H, \  g_0 \in \mathbb R
    \end{empheq}
\end{subequations}
\end{proposition}
\begin{proof}
Note that we can replace $\inf_{P \in \mathcal P} P [f(x, \xi) \leq 0] \geq 1 - \alpha$ with the equivalent $\sup_{P \in \mathcal P} P [f(x, \xi) > 0] \leq \alpha$.
Later, we will introduce an inner approximation scheme based on this result.

Next, we can rewrite the probability with an indicator function
$$
P [f(x, \xi) > 0] = \mathbb E_P [\mathbb 1\big(f(x, \xi) > 0\big)].
$$
The indicator function fulfills the assumptions on the constraint function of the strong duality result of \cite{zhuKernelDistributionallyRobust2020}, with the Slater condition trivially satisfied.
Using their Theorem 3.1 we can rewrite $\sup_{P \in \mathcal P} \mathbb E_P [\mathbb 1\big(f(x, \xi) > 0\big)]$ as
\begin{alignat}{2}
& \min_{g \in \mathbb{R}, g\in \mathcal H} \quad && g_0 + \frac{1}{N}\sum_{i=0}^N g(\xi_i) + \varepsilon ||g||_{\mathcal H}\\
& \text{subject to} \quad && \mathbb 1\big(f(x, \xi) > 0\big) \leq g_0 + g(\xi) \quad \forall \xi \in \Xi.
\end{alignat}
The result follows by plugging this expression into \eqref{eq:set_z}.
\end{proof}
Solving MMD-DRCCP using this exact reformulation is intractable in practice. As a consequence, we will investigate a convex CVaR approximation, which can be solved with off-the-shelf solvers. Note that, unlike the approximations in \cite{hotaDataDrivenChanceConstrained2018, chenDataDrivenChanceConstrained, Weijun2018}, our approximation still holds for general nonlinear $f(x,\xi)$ and only requires mild assumptions on the dependence on the uncertainty $\xi$.
For more technical details, we refer to \cite[Theorem~3.1]{zhuKernelDistributionallyRobust2020}.

\begin{figure}[t!]
    \centering
    \includegraphics[scale=0.85]{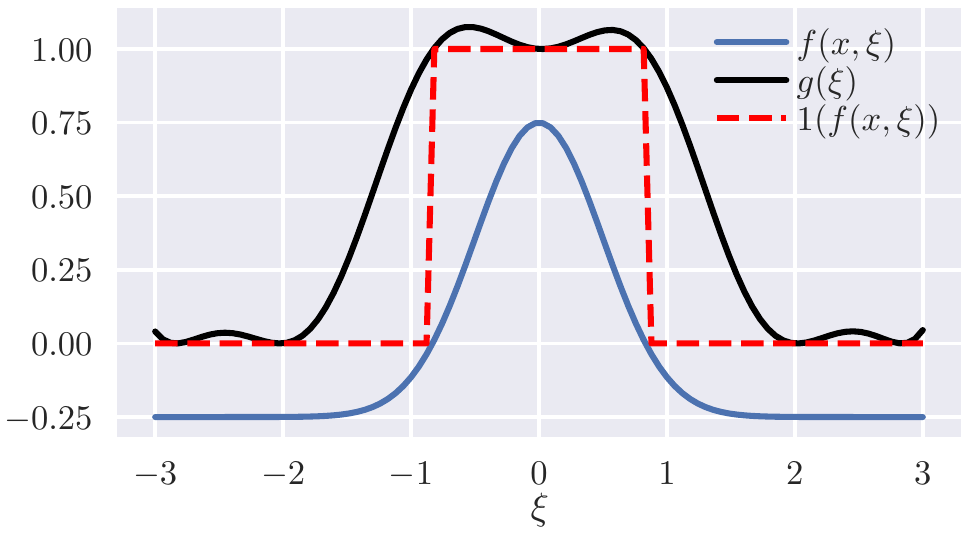}
    \caption{Visualization of the majorization of the RKHS-function $g(\xi)$ examplarily for a Gaussian constraint function $f$ and fixed $x$.}
    \label{fig:majorization}
\end{figure}

\subsection{CVaR approximation}
\label{subsec:CVaR_Relax}
In this section we present a convex inner approximation of the feasible set \eqref{eq:exact_ref} based on the CVaR.
First, note that we can rewrite the chance constraint in \eqref{eq:ccp} equivalently in terms of the Value-at-Risk (VaR) which is defined as
\begin{align*}
    \operatorname{VaR}^{P_0}_{1-\alpha}[f(x,\xi)] = \inf \{ t \in \mathbb{R}: P_0[f(x,\xi) \leq t] \geq 1-\alpha \},
\end{align*}
where $f(x,\xi)$ is to interpreted as a random variable. With this definition, it is straightforward to observe
\begin{align}
    \operatorname{VaR}^{P_0}_{1-\alpha}[f(x,\xi)] \leq 0 \ \ \Leftrightarrow \ \ P_0[f(x,\xi) \leq 0] \geq 1-\alpha.
\end{align}
While generally the VaR constraint is non-convex even for convex constraint functions $f(x,\xi)$, it is has been shown by \cite{nemirovskiConvexApproximationsChance}, building on the idea of \cite{TyrrellRockafellar}, that the tightest conservative convex approximation of VaR is given by the conditional value-at-risk (CVaR) defined as
\begin{align}
    \operatorname{CVaR}_{1-\alpha}^{P_0}[f(x,\xi)] = \inf_{t \in \mathbb{R}} \mathbb{E}_{P_0}[ [f(x,\xi) + t]_+ - t \alpha],
\end{align}
where $[\cdot]_+ = \max[0, \cdot]$ denotes the maximum operator.
Using this result we can express the tightest convex conservative approximation of our distributionally robust constraint \eqref{eq:drcc} as
\begin{equation}
    \sup_{P \in \mathcal P} \CVaR_{1-\alpha}^{P}[f(x, \xi)] = \sup_{P \in \mathcal P} \inf_{t\in \mathbb R} [\mathbb E_P[ [f(x, \xi) + t]_+ ] - t\alpha] \leq 0
    \label{eq:1}
\end{equation}
The following Lemma is based on the stochastic min-max equality theorem of \cite{shapiroMinimaxAnalysisStochastic2002} and shows that we can exchange the supremum and infimum in \eqref{eq:1}.
\begin{lemma}
\label{lemma:minmax}
    Let $\Xi \subset \mathbb{R}^m$, and $f: \mathcal{X} \times \mathbb{R}^m \rightarrow \mathbb{R}$ such that $\xi \mapsto f(x,\xi)$ is bounded on $\Xi$, $\forall x \in \mathcal{X}$.
    Then,
\begin{equation}
     \sup _{P \in \mathcal{P}} \inf _{t \in \mathbb{R}} \mathbb{E}_{P}[ [f(x, \xi) + t]_+ ] - t\alpha] = 
     \inf _{t \in \mathbb{R}} \sup _{P \in \mathcal{P}} \mathbb{E}_{P}[ [f(x, \xi) + t]_+ ] - t\alpha].
     \label{eq:2}
\end{equation}
\end{lemma}
\begin{proof}
The proof is identical to the one of Lemma IV.2.\ of \cite{hotaDataDrivenChanceConstrained2018} for Wasserstein ambiguity sets, using that the MMD-based uncertainty set $\mathcal{P}$ is also weakly compact \cite{zhuKernelDistributionallyRobust2020}.
\end{proof}

Applying the generalized duality result of \cite{zhuKernelDistributionallyRobust2020} to the supremum in the RHS of \eqref{eq:2} yields the following result which provides an expression for the feasible set of the CVaR approximation of MMD-DRCCP.

\begin{proposition}
Let the conditions of Lemma~\ref{lemma:minmax} be fulfilled. Then the distributionally robust CVaR (DR-CVaR) constraint $\sup_{P \in \mathcal P} \CVaR_{1-\alpha}^{P}[f(x, \xi)] \leq 0$ is equivalent to $x \in Z_{\CVaR}$, where
    \begin{subequations}
        \label{eq:z_cvar}
        \begin{empheq}[
        left={Z_{\CVaR}\coloneqq \empheqlbrace x \in \mathcal{X}:},
        right=\empheqrbrace]
        {align}
            & g_0 + \frac{1}{N}\sum_{i=1}^N g(\xi_i) + \varepsilon \| g \|_{\mathcal{H}} \leq t \alpha 
            \label{subeq:cvar_risk}\\
            & [f(x,\xi) + t]_+ \leq g_0 + g(\xi) \ \ \forall \xi \in \Xi  \label{eq:semi-infinite}\\
            & g \in \mathcal{H}, \quad t \in \mathbb{R}
        \end{empheq}
    \end{subequations}
\end{proposition}
\begin{proof}
    The proof follows directly from Lemma~\ref{lemma:minmax} and the generalized duality result of \cite{zhuKernelDistributionallyRobust2020} applied to the inner supremum in \eqref{eq:2}
    \begin{align*}
        &\sup_{P \in \mathcal P} \CVaR_{1-\alpha}^{P}[f(x, \xi)] \\
        &=\inf _{t \in \mathbb{R}} \sup _{P \in \mathcal{P}} \mathbb{E}_{P}[ [f(x, \xi) + t]_+  - t\alpha]  \\
        &= \inf_{g_0, t \in \mathbb{R}, g \in \mathcal{H}} g_0 + \frac{1}{N} \sum_{i=1}^N g(\xi_i) + \varepsilon \|g \|_{\mathcal{H}} - t \alpha \\
        & \quad \text{s.t.} \ [f(x, \xi) + t]_+  \leq g_0 + g(\xi) \ \ \forall \xi \in \Xi. 
    \end{align*}
    The requirement of the infimum over $(t,g_0,g)$ to be $\leq 0$ is equivalent to requiring that there exists any $t,g_0 \in \mathbb{R}$ and $g \in \mathcal{H}$ such that  $g_0 + 1/N \sum_{i=1}^N g(\xi_i) + \varepsilon \|g \|_\mathcal{H} - t \alpha \leq 0$ and thus the result follows.
\end{proof}

We now use a constraint sampling approximation to the infinite constraint \eqref{eq:semi-infinite} following \cite{zhuKernelDistributionallyRobust2020}. Namely, the constraint \eqref{eq:semi-infinite} is replaced by its empirical version
\begin{align}
    \label{eq:constr_sampled}
    [f(x,\xi_i) + t]_+ \leq g_0 + g(\xi_i), \quad i=1,\ldots,N.
\end{align}
However, unlike the convergence analysis for semi-infinite program discretization such as in \cite{roysetRateConvergenceAnalysis2012}, we later provide a finite-sample guarantee with a convergence rate of (at least) $\oneOverRootN$ independent of the dimensionality of the $\xi$ variable.
Using the robust representer theorem of \cite{zhuKernelDistributionallyRobust2020}, we can express the RKHS function $g$ in terms of finite dimensional parameters $\gamma \in \mathbb{R}^{N}$. Let $K$ denote the kernel Gram matrix with $K_{i,j} = k(\xi_i, \xi_j)$, then we can write the sample approximation of \eqref{eq:z_cvar} as
\begin{subequations}
    \label{eq:approx_sem_inf}
    \begin{empheq}[
    left={\hat{Z}_{\CVaR} \coloneqq \empheqlbrace x \in \mathcal{X} :},
    right=\empheqrbrace]
    {align}
        & g_0 + \frac{1}{N}\sum_{i=1}^N (K\gamma)_i + \varepsilon \sqrt{\gamma^T K \gamma } \leq t \alpha \label{eq:representer1}\\
        &  [f(x,\xi_i) + t]_+ \leq g_0 + (K \gamma)_i, \quad i=1,\ldots,N \\
        & g_0 \in \mathbb{R}, \ \gamma \in \mathbb{R}^{N}, \ t \in \mathbb{R} \label{eq:representer3}
    \end{empheq}
\vspace*{0.05cm}
\end{subequations}

In contrast to DRCCP using Wasserstein ambiguity sets that requires ad-hoc reformulation techniques depending on various forms of $f$ (cf. Table~\ref{tab:drccp_comp} and \cite{hotaDataDrivenChanceConstrained2018}),
we simply use \eqref{eq:approx_sem_inf} to treat all forms of $f$.
This generality is due to that RKHSs are universal function approximators~\cite{steinwartSupportVectorMachines2008,wendlandScatteredDataApproximation2004}.
This has also been empirically demonstrated in large-scale adversarial learning tasks with deep neural networks, as reported in \cite{zhuAdversariallyRobustKernel2021}.
When the function $f$ is convex in the decision variable $x$, the problem~\eqref{eq:approx_sem_inf} is a convex kernel approximation problem and can be solved with an off-the-shelf convex optimization solver.

\begin{remark}
Similar to DRCCP with Wasserstein ambiguity sets,
within our MMD-DRCCP framework, one can also obtain so-called exact reformulations using certain kernel choices.
For example, for piece-wise linear constraints supported on a closed convex cone, one can derive tractable CVaR approximations by choosing linear kernels.
However, such reformulations only apply to ad-hoc $f$ classes and restrict kernel choices to less expressive kernels, e.g., linear kernel.
In those cases, MMD is not a metric since it only detects differences in the first moment.
Therefore, we favor the single general approximation \eqref{eq:approx_sem_inf} over such exact reformulations.
\end{remark}
\begin{remark}
    Note that the convexity of the inner CVaR approximation is inherited from the convexity of the function $f$ with respect to $x$, i.e., our reformulation technique is \emph{convexity-preserving} but can also treat non-convex $f$ in practice.
    We later demonstrate this in an optimal control problem with non-convex constraints.
\end{remark}

\section{Finite-sample guarantee for constraint satisfaction}
\label{sec:guarantee}
Since we rely on the approximate formulation \eqref{eq:approx_sem_inf} to treat general nonlinear constraints, an important task is to quantify the approximation error.
Furthermore, we are interested in finite-sample analysis instead of asymptotic consistency results such as reported in \cite{cherukuriConsistencyDistributionallyRobust2020} for Wasserstein ambiguity sets, since the former is more informative for quantifying robustness against estimation error.
Our result in this section shows that, using \eqref{eq:approx_sem_inf}, we can solve the MMD-DRCCP with a finite-sample guarantee for constraint satisfaction at the $\oneOverRootN$ rate independent of dimensions.
This concentration result gives the foundation for the distributional robustness of \eqref{eq:approx_sem_inf}.
Existing analysis for MMD ambiguity sets (see e.g., \cite{lamComplexityFreeGeneralizationDistributionally2021} for references) only concerns the exact solution to \eqref{eq:drcc}, which is unavailable in practice.
To date, the only available algorithm to solve DRO problems with MMD ambiguity sets is that of \cite{zhuKernelDistributionallyRobust2020}, whose statistical guarantee is not yet established.
In contrast, we now show the first finite-sample guarantee of constraint satisfaction for the proposed practical algorithm. Furthermore, the finite-sample analysis in this section only assumes mild boundedness for the constraint function $f$, which is significantly less restrictive than the Wasserstein ambiguity sets.
In the proofs, we make the mild assumptions that the function
$f(x, \cdot)$ and $g$ are bounded in infinity norm, i.e., $\exists M>0: \forall \xi, |f(x, \xi)| \leq M/2, |g(\xi)| \leq M/2$.
For conciseness, we simple write $g$ instead of $g_0+g$ since one can re-define a new RKHS to include the constant $g_0$ term.
\begin{proposition}
    [Finite-sample guarantee for MMD-DRCCP constraint satisfaction]
    \label{thm:guarantee}
    Let $(\xh, \gh)$ be a pair of the solution to the CVaR approximation\eqref{eq:approx_sem_inf}.
Suppose the kernel is bounded in the sense that $\exists C>0, \sup_x|k(x,x)|\leq C$, and the radius of the MMD ambiguity set satisfies
$\rho \geq \sqrt{\frac{C}{N}} + \sqrt{\frac{2C\log(1/\delta)}{N}}$ (see \eqref{eq:rate}).
Then,
$\xh$ is an $\confidTerm$-approximate feasible solution to the MMD-CVaR approximation~\eqref{eq:z_cvar} with probability at least $1-\delta$, i.e.,
\begin{equation}
    \label{eq:result_pf_cvar}
    \operatorname{CVaR}^{P_0}_{1-\alpha}[f(\xh, \xi)] \leq \confidTerm.
\end{equation}
\end{proposition}
\begin{proof}
We first expand the left-hand-side of \eqref{eq:result_pf_cvar},
\begin{equation}
    \operatorname{CVaR}^{P_0}_{1-\alpha}[f(\xh, \xi)] = 
     \inf_t\frac{1}{\alpha}\Etrue [f(x,\xi) + t]_+   -t.
\end{equation}
Note that the expectation term in CVaR can be written as
\begin{equation}
    \begin{aligned}
        \label{pf:split_1}
        \Etrue[f(\xh,\xi) + t]_+
        &=
        \Etrue \gh + \Etrue( [f(\xh,\xi) + t]_+ - \gh ) 
    \end{aligned}
\end{equation}

For the first term in \eqref{pf:split_1}, we have
\begin{multline}
    \Etrue \gh = \Ehat\gh + \int{\gh}\ d{(P_0 - \pemp)}
    \leq \Ehat\gh + \MMD(P_0,\pemp)\cdot \hnorm{\gh}
    \leq \Ehat\gh + \rho \hnorm{\gh},
\end{multline}
where the first inequality is simply Cauchy-Schwarz and the second inequality is due to the condition that $\rho \geq \MMD_\rkhs (P_0, \pemp)$ with high probability as noted in \eqref{eq:rate}.
For ease of notation, let
\begin{equation*}
    \E := \confidTerm.
\end{equation*}

For the second term in \eqref{pf:split_1}, we apply standard concentration results, namely, McDiarmid inequality
\begin{equation}
        \Etrue( [f(\xh,\xi) + t]_+ - \gh )
        \leq
        \Ehat( [f(\xh,\xi) + t]_+ - \gh )
        +\E \leq 0 + \E
\end{equation}
For the last inequality above, we exploited the relationship $[f(\xh,\xi) + t]_+ \leq \gh$ that holds at the empricial sample $\xi_i$ due to the constraints in \eqref{eq:constr_sampled}.
Plugging both inequalities back into \eqref{pf:split_1}, we obtain
\begin{equation}
        \Etrue[f(\xh,\xi) + t]_+
        \leq
        \Ehat\gh + \rho \hnorm{\gh}
        + \E.
\end{equation}
Combining the result above with the relationship in \eqref{subeq:cvar_risk} of the CVaR approximation of the MMD-DRCCP, we arrive at the proposition statement.
\end{proof}

Using the relationship between CVaR and chance constraints, this result directly translates to the following.
\begin{corollary}
    Under the same assumption as Proposition~\ref{thm:guarantee}, with probability at least $1-\delta$,
    \begin{equation}
        \label{eq:result_pf_cc}
        P\left(f(\xh, \xi) \leq \confidTerm\right) \geq 1-\alpha
    \end{equation}
\end{corollary}

Our guarantees above state that the MMD-DRCCP solution approximately satisfy the DRCCP constraint with a rate of $\oneOverRootN$ independent of dimensions.
This can further motivate a constraint back-off design that adds to the left-hand-side of 
\eqref{subeq:cvar_risk} and \eqref{eq:representer1} by the $\confidTerm$ term to guarantee safety.
        \label{{}}
That way, the $\confidTerm$ term will no longer appear in our guarantee statements \eqref{eq:result_pf_cvar} and \eqref{eq:result_pf_cc}.
However, due to the general conservatism of DR-CVaR approximations and ambiguity set sizes, we observe the current MMD-DRCCP alone is sufficient in practice.

\begin{remark}
    The convergence guarantees presented above can be further made uniform w.r.t. the decision variables $(x, g)$ using uniform convergence results for empirical processes~\cite{vandervaartWeakConvergence1996}.
    In addition, we leave further refinement of the convergence rate beyond $\oneOverRootN$ for future work.
\end{remark}

\section{Numerical examples}
\label{sec:num_exp}
In the following section, we present numerical results for our MMD-DRCCP algorithms. Within a chance constrained portfolio optimization problem, we provide empirical evidence to support our theoretical finite sample guarantee (Proposition~\ref{thm:guarantee}). Moreover we show that the bootstrap construction improves on the MMD-rate-based ambiguity set in terms of providing a less conservative ambiguity size.
Enabled by our theory, we simply use a Gaussian kernel $k(x, y) = \exp(-\frac{1}{2\sigma^2}||x-y||^2_2)$ with the bandwidth $\sigma$ set via the median heuristic \cite{garreauLargeSampleAnalysis2018} for all experiments.
We further emphasize that we do not exploit any ad-hoc transformations such as convex conjugate of certain specific $f$ functions.
We simply use our general approximation scheme \eqref{eq:approx_sem_inf} as a universal technique across all function classes.

\subsection{Chance-constrained portfolio optimization}
\begin{figure}[t]
    \centering
    \includegraphics[scale=0.85]{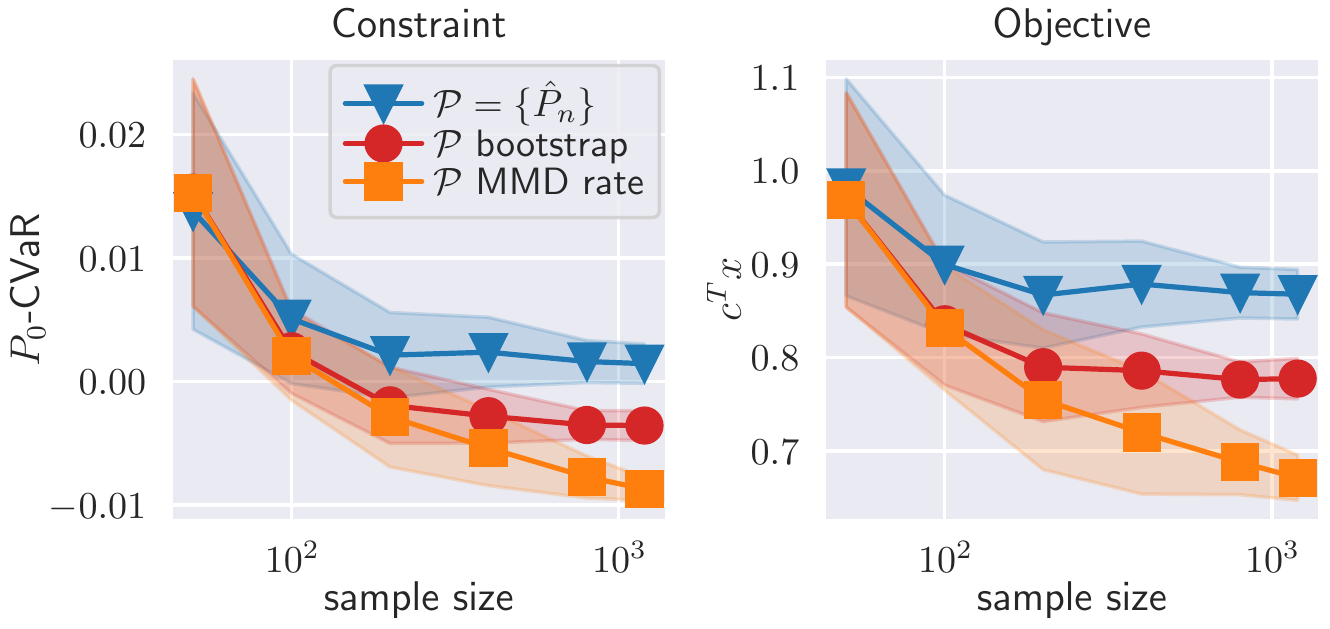}
    \caption{CVaR relaxation for chance constrained portfolio optimization. Lines and shaded regions denote the mean and standard deviation over 16 runs respectively.
    }
    \label{fig:cvar_exp}
\end{figure}
We consider a chance-constraint portfolio optimization problem, where we want to optimally allocate resources $x \in \Delta := \{x\in \mathbb{R}_{0,+}^3: \sum_{i=1}^3 x_i \leq 1\}$ to investments with returns $c = (1, 1.5, 2)^T \in \mathbb{R}^3$ in presence of a chance constraint depending on uncertain variables $\xi\sim \mathcal{N}(0,\operatorname{diag}[0.5, 1, 1.5])$:
\begin{align}
     \max_{x \in \Delta}  c^T x  \quad \text{s.t.} \quad   P [f(x,\xi) \leq 0] \geq 1 - \alpha, 
     \label{eq:ccp-portfolio}
\end{align}
where the constraint function is given by the nonlinear function $f(x,\xi) = (\xi^T x)^2 - 1$.
The CVaR approximation of the constraint can be written as $\operatorname{CVaR}_{1-\alpha}^{\pemp}[f(x;\xi)] \leq 0$ and the corresponding MMD-DRCCP is then given by 
\begin{align}
     \max_{x \in \Delta}  c^T x \quad \text{s.t.} \quad \sup_{P \in \mathcal{P}}\operatorname{CVaR}_{1-\alpha}^{P}[f(x;\xi)] \leq 0.
\end{align}

We construct MMD ambiguity sets using the MMD convergence rate $\eqref{eq:rate}$ as well as our bootstrap construction (see Algorithm~\ref{alg:bootstrap}). We solve the problem for different sample sizes via the convex reformulation \eqref{eq:representer1}-\eqref{eq:representer3} using CVXPY \cite{diamondCVXPYPythonEmbeddedModeling2016} and compare the results to a (non-robust) CVaR approximation of \eqref{eq:ccp-portfolio} (equivalent to DR-CVaR with ambiguity set $\mathcal{P} = \{\pemp\}$). At test time we sample $10^6$ data points from the true distribution in order to estimate $\operatorname{CVaR}^{P_0}_{1-\alpha}[f(\xh, \xi)]$. We observe in Figure~\ref{fig:cvar_exp} that, while the non-robust solution fails to fulfill the CVaR constraint for the population distribution across all numbers of training samples, both MMD-DRCCP solutions fulfill the constraint for training sample sizes $\geq 100$. Moreover, we observe that the bootstrap version yields a tighter estimate of the ambiguity set and thus a less conservative solution which allows for larger objective values as observed in the right panel of Figure~\ref{fig:cvar_exp}.

\subsection{Distributionally robust stochastic MPC with nonlinear constraints}
In this example, we highlight a tube-based MPC problem with linear dynamics but nonlinear non-convex constraints.
For a detailed explanation of the application of distributionally robust CC to tube-based MPC, we refer to \cite{nemmourApproximateDistributionallyRobust2021}.
We consider the problem of controlling a double-integrator system $f_{dyn}(x, u)$ with additive noise subject to a constraint given in the form of a non-convex SVM classifier.
The optimal control problem (OCP) with horizon $H$ and quadratic cost $J(\mathbf x, \mathbf u)$ is subject to the constraints
\begin{equation*}
    x^l_{k+1} = f_{dyn}(x^l_k, u^l_k),
 \sup_{\text{MMD}(P, \hat{P})\leq \varepsilon}\operatorname{CVaR}_{1-\alpha}^{P}[h(x^l_k)] \leq 0,
\end{equation*}
where $l$ denotes the current iteration in the MPC-loop, $\mathbf{u}^l = [u^l_0, \dots, u^l_{H-1}]$ the actions, $\mathbf{x}^l = [x^l_0, \dots, x^l_{H-1}]$ the states, and $k=0, \dots, H-1$. 
We solve the OCP using MMD-DRCCP with bootstrap ambiguity sets and visualize the resulting closed-loop trajectories with high constraint satisfaction in Figure~\ref{fig:drcc_mpc}. Note that related methods for MPC with Wasserstein ambiguity sets \cite{zhongDatadrivenDistributionallyRobust2021} are restricted to affine constraints and thus not applicable to this problem, which highlights the greater generality of our approach.
\begin{figure}[t!]
    \centering
    \includegraphics[scale=0.85]{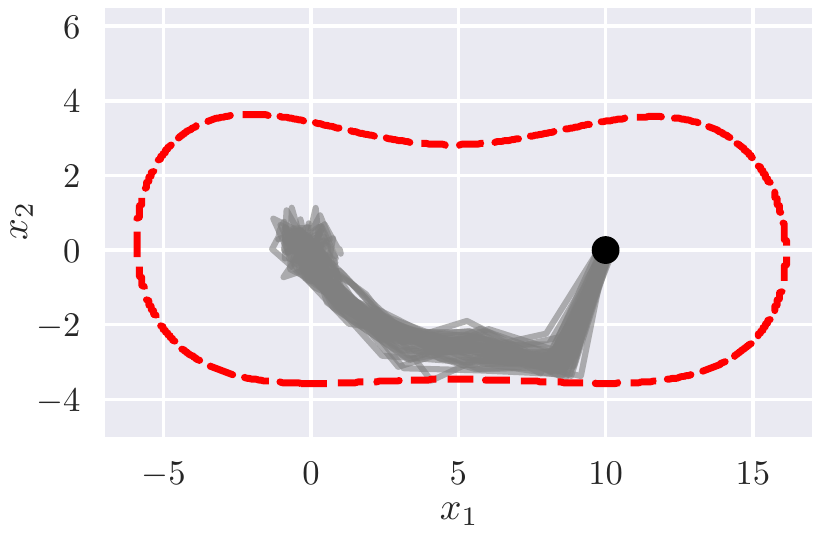}
    \caption{Tube-based MPC with 40 samples of the additive disturbance $\mathcal N(0, 0.2)$. We add small uniform noise on the initial state $(10, 0)$ and plot the 30 different resulting trajectories in grey, visualizing the dynamics tube resulting from the MMD-DRCCP.
    }
    \label{fig:drcc_mpc}
\end{figure}
\section{Further Related Work}
\label{sec:rel_work}
DRCCPs have recently attracted significant attention in the stochastic programming and control community \cite{coulson2021distributionally, zhongDatadrivenDistributionallyRobust2021}. A significant focus has been laid on Wasserstein ambiguity sets \cite{Weijun2018, chenDataDrivenChanceConstrained, hotaDataDrivenChanceConstrained2018, jiDataDrivenDistributionallyRobust, ho-nguyenDistributionallyRobustChanceConstrained2020}, for which the strong duality result of \cite{gaoDistributionallyRobustStochastic2016} plays a fundamental role.
As chance constraints are generally non-convex even for convex constraints \cite{nemirovskiConvexApproximationsChance}, a common approach to the problem relies on approximating the chance constraint via the conditional Value-at-Risk (CVaR) \cite{TyrrellRockafellar}, which has been shown to provide the tightest conservative convex approximation \cite{nemirovskiConvexApproximationsChance}.
As the dual formulation of the Wasserstein DRCCP contains a constraint involving a supremum over the uncertain variable, many works restrict their scope to constraint functions affine in the uncertain variable for which the supremum can be carried out in closed form. A notable exception is given by \cite{hotaDataDrivenChanceConstrained2018}, which only assumes the constraints to be concave in the uncertainty and proposes solving the resulting semi-infinite program with a cutting-plane algorithm to compute an approximate solution. 
Compared to probability metrics such as the Wasserstein distance, $\phi$-divergences require absolute continuity of the two considered distributions with respect to each other, thus limiting the generality of these approaches.
Nonetheless, we refer to \cite{ben-talRobustSolutionsOptimization2013,namkoongVariancebasedRegularizationConvex,lamRecoveringBestStatistical2016,jiangDatadrivenChanceConstrained2016} for other technical details for $\phi$-divergence-DRO.
Statistical guarantees for Wasserstein DRO \cite{MohajerinEsfahani2018} provide a basis for consistency results in \cite{cherukuriConsistencyDistributionallyRobust2020}. Those authors show that the DRCCP with Wasserstein ambiguity set converges to the CCP from above as the samples size goes to infinity.
However, unlike the finite-sample analysis in this paper, those consistency results do not provide error bounds for solutions computed using finitely many samples, which is the key to certifying robustness.
They also assume access to the optimal solution to the original program, which may not be available depending on the constraint function $f$.
In the context of CCPs, kernel methods have been used previously to estimate the unknown distribution over the uncertainty variables via kernel density estimation to obtain a deterministic problem \cite{caillauSolvingChanceConstrained2018,keilBiasedKernelDensity2020}.
A few works in stochastic control \cite{thorpeDataDrivenChanceConstrained2022,zhuWorstCaseRiskQuantification2020,zhuKernelMeanEmbedding2020,gopalakrishnanSolvingChanceConstrainedOptimization2021} have considered mean embeddings of the chance constraints to obtain approximations.
None of those works contains algorithms that can solve MMD-constrained DRCCP or provide finite-sample guarantees like in this paper.
To the best of our knowledge, this work is the only work to solve DRCCP with MMD ambiguity sets in a principled way, thus utilizing the unique advantage of MMD over Wasserstein-based approaches, e.g., bootstrap MMD ambiguity sets, approximating nonlinear constraints, and favorable finite-sample guarantees. 

\section{Conclusion}
\label{sec:conclusion}
In this work, we presented distributionally robust chance-constrained optimization with MMD ambiguity sets. 
Leveraging recent results in kernel methods for robust machine learning, we provided a practical algorithm for 
distributionally robust conditional Value-at-Risk constraints with finite-sample constraint satisfaction guarantees.
Different from methods based on Wasserstein ambiguity sets, we give a practical bootstrap scheme that enables a priori computation of suitably sized ambiguity sets. Moreover, our method can be applied to general nonlinear constraint functions and thus parts with the strong assumptions of other recently proposed frameworks.
For future work, we plan to derive sharper finite-sample guarantees for constraint satisfaction and optimality and explore more applications to robust nonlinear control problems.

\section{Acknowledgement}
We thank Ren\'e Henrion for his helpful feedback.

\printbibliography

\end{document}